\documentclass{amsart}[12pt]
\usepackage{graphicx,graphics, amsmath, amsthm, amscd, amsfonts}

\usepackage{latexsym}
\makeatletter \oddsidemargin.9375in \evensidemargin \oddsidemargin
\newtheorem{theorem}{Theorem}[section]
\newtheorem{lemma}[theorem]{Lemma}

\newtheorem{corollary}[theorem]{Corollary}
\theoremstyle{definition}
\newtheorem{definition}[theorem]{Definition}

\newtheorem{remark}[theorem]{Remark}
\numberwithin{equation}{section}

\begin{document}
\Large
\title[Hyperinner Product Spaces]
{Hyperinner Product Spaces}

\author[Ali Taghavi, Roja Hosseinzadeh and Hamid Rohi ]{Ali Taghavi$^1$, Roja Hosseinzadeh $^2$ and  Hamid Rohi $^3$}

\address{{ $^{1}$Department of Mathematics, Faculty of Mathematical Sciences,
 University of Mazandaran, P. O. Box 47416-1468, Babolsar, Iran.}}

\email{$^{1}$Taghavi@umz.ac.ir, $^{2}$ro.hosseinzadeh@umz.ac.ir,
$^{3}$h.rohi@umz.ac.ir}



\begin{abstract}\large
In this paper, we introduce the concept of inner product on weak hypervector spaces and prove some results about them.
\end{abstract} \maketitle

\section{\textbf{Introduction}}

\noindent The concept of hyperstructure was first introduced by
Marty $[3]$ in 1934 and has attracted attention of many authors in
last decades and has constructed some other structures such as
hyperrings, hypergroups, hypermodules, hyperfields, and
hypervector spaces. These constructions have been applied to many
disciplines such as geometry, hypergraphs, binary relations,
combinatorics, codes, cryptography, probability, and etc. A
wealth of applications of this concepts is given in $[1-3]$ and $[15]$.
\par In 1988, the concept of hypervector space was
first introduced by Scafati-Tallini.  She later considered more properties of such spaces. Authors, in $[4]$ and $[6-11]$ considered hypervector spaces in viewpoint of analysis. In mentioned papers, authors introduce concepts such as dimension of hypervector spaces, normed hypervector spaces, operator on these spaces and another important concepts. In this paper, we introduce the concept of inner product on a weak hypervector space. Authors in $[5]$ introduced the definition of real inner product on a hypervector space with hyperoperations sum and scalar product over the real hyperfield. We extend this concept to complex inner product for hypervector spaces with usual sum and hyperoperation scalar product on real or complex fields.
\par Note that the hypervector spaces used in this
paper are the special case where there is only one hyperoperation,
the external one, all the others are ordinary operations. The
general hypervector spaces have all operations multivalued also in
the hyperfield (see $[15]$). In throughout of this paper assume that $X$ and $Y$ are weak hypervector spaces over a field $F$.
\section{\textbf{hyperinner product spaces}}
\begin{definition}
($[14]$) A $weak$ or $weakly$ $distributive$ $hypervector space$ over
a field $F$ is a quadruple ($X$,+,$o$,$F$) such that $(X,+)$ is an
abelian group and $o:F\times X\longrightarrow P_*(X)$ is a
multivalued product such that\\
1- $\forall a\in F,\forall x,y\in X,~[ao(x+y)] \cap [aox+aoy]\neq
\emptyset$;\\
2- $\forall a,b\in F,\forall x\in X,~[(a+b)ox]
\cap [aox+box]\neq \emptyset$;\\
3- $\forall a,b\in F,\forall x\in X,~ao(box)=(ab)ox$;\\
4- $\forall a\in F,\forall x\in X,~ao(-x)=(-a)ox=-(aox)$;\\
5- $\forall x\in X,~x\in 1ox.$\\
The properties 1 and 2 are called weak right and left distributive
laws, respectively. Note that the set $ao(box)$ in 3 is of the
form $\cup_{y\in box}aoy$.
\end{definition}
\begin{definition} Let $X$ is a weak hypervector space over $F$, $a\in F$
and $x\in X$. Essential point of $aox$, that we denote it by $e_{aox}$, for $a\neq 0$ is the element of
$aox$ such that $x\in a ^{-1}oe_{aox}$. For $a =0$, we define
$e_{aox}=0$.
\end{definition}
\begin{remark}
Note that $e_{aox}$ is not unique, necessarily.
 Hence we denote the set of all essential points by
$E_{aox}$. When in this note we use $e_{aox}$ in an equation, we
intend is any element of $E_{aox}$.
\end{remark}
\begin{definition} Let $X$ be a weak hypervector space over $F$ with the following
properties

1- $E_{a_1ox}+E_{a_2ox})\cap E_{(a_1+a_2)ox}\neq \emptyset,~\forall x\in X,~\forall
a_1,a_2\in F$,

2- $(E_{aox_1}+E_{aox_2})\cap E_{ao(x_1+x_2)}\neq \emptyset,~\forall x_1,x_2\in X,~\forall
a\in F$.

 Then $X$ is called normal weak hypervector
space.
\end{definition}
\begin{lemma}
Let $X$ is a weak hypervector space over $F$, $a,b\in F$ and $x \in X$. Then the following properties hold.

1- $x\in E_{1ox}$.

2- If $b \neq 0$, then $aoe_{box}=abox$.

3- $E_{-aox}=-E_{aox}$.

4- If $a \neq 0$, then there exists an $y \in X$ such that $x \in E_{aoy}$.

5- If $X$ is normal, then $E_{aox}$ is singleton.
\end{lemma}
\begin{lemma}
Let $X$ be a weak hypervector space over $F$. $X$ is normal if
and only if
$$e_{a_1ox}+e_{a_2ox}=e_{(a_1+a_2)ox},~\forall x\in X,~\forall
a_1,a_2\in F,$$
$$e_{aox_1}+e_{aox_2}=e_{ao(x_1+x_2)},~\forall x_1,x_2\in X,~\forall
a\in F.$$
\end{lemma}
Authors in $[5]$ introduced the definition of inner product on a hypervector space with hyperoperations sum and scalar product over the real hyperfield. We restate it for a hypervector space with usual sum and hyperoperation scalar product over the real field.
\begin{definition} $[5]$ Let $X$ be a hypervector space over the real field.
 An inner product on $X$ is a mapping $(.,.):X\times X
\rightarrow \mathbb{R}$ such that for every $a\in \mathbb{R}$ and $x,y,z\in X$ we have

 1-  $  $  $ (x,x)>0$ for $x\neq 0$;

 2-   $  $  $(x,x)=0 \Leftrightarrow x=0$;

 3-   $  $  $ (x+y,z)=(x,z)+(y,z)$;

 4-   $  $  $ (y,x)=(x,y)$;

 5-   $  $  $ \mathrm{sup}( aox,y)=a(x,y)$,

 where $ \mathrm{sup}( aox,y)=\{(z,y): z \in aox \}$.
\end{definition}
By the above definition we have the following lemma:
\begin{lemma} Let $X$ be a hypervector space over the real field. Then $ \mathrm{sup}(aox,y)=(e_{aox},y)$, for every $a\in \mathbb{R}$ and $x,y\in X$.
\end{lemma}
\begin{proof}
Since $e_{aox}\in aox$, we have
 $$(e_{aox},y)\leq \mathrm{sup}(aox,y)=a(x,y). \leqno {(1)}$$
Since $x \in 1ox$, we have $ (x,y) \in (1ox,y)=(a^{-1}o(aox),y)$ and so
$$(x,y) \leq \mathrm{sup}(a^{-1}o(aox),y)= \mathrm{sup}(a^{-1}oe_{aox},y)=a^{-1}(e_{aox},y). \leqno {(2)}$$
If $a>0$, then $(1)$ and $(2)$ imply that $(e_{aox},y)=a.(x,y)$.
If $a<0$, we have
$$\mathrm{sup}(aox,y)=\mathrm{sup}((-a)o(-x),y)=(e_{(-a)o(-x)},y)=(e_{aox},y).$$
If $a=0$, then assertion follows from Result 4.5 in $[5]$ and $e_{0ox}=0$.
\end{proof}
The defined inner product in Definition 3.1 is a real mapping for real hypervector spaces. The existence of essential points and Lemma 3.2 are the motivation of introducing a new definition of inner product for arbitrary hypervector spaces over arbitrary fields such that in real case is equivalent to Definition 3.1.

Next assume that the field $F$ is a real or complex field.
\begin{definition}
 An inner product on $X$ is a mapping $(.,.):X\times X
\rightarrow F$ such that for every $a\in F$ and $x,y,z\in X$ we have

 1-  $  $  $ (x,x)>0$ for $x\neq 0$;

 2-  $  $  $(x,x)=0 \Leftrightarrow x=0$;

 3-  $  $  $ (x+y,z)=(x,z)+(y,z)$;

 4-  $  $  $ (y,x)=\overline{(x,y)}$;

 5-  $  $  $ (e_{aox},y)=a(x,y)$;

 6-  $  $  $ (u,u) \leq (x,x)$ for every $u \in 1ox$.

 $X$ with an inner product is called a hyperinner product space.
\end{definition}
\begin{lemma}
Let $(.,.)$ be an inner product on $X$. Then the following statements are hold.

1-   $  $  $(0,x)=(x,0)=0$

2-   $  $ $(-x,y)=(x,-y)=-(x,y)$

3-   $  $ $(x,e_{aoy})=\overline{a}(x,y)$

4-   $  $  $ (u,u) \leq a^2(x,x)$ for every $u \in aox$.
\end{lemma}
\begin{proof}
1- $(0,x)= (e_{0ox},x)=0(x,x)=0$, $(x,0)=\overline{(0,x)}=0$.

2- $  $$(-x,y)=(e_{1o(-x)},y)=(e_{(-1)ox},y)=-(x,y)$, \\
$(x,-y)=\overline{(-y,x)}=\overline{-(y,x)}=-(x,y)$.

3- $  $$(x,e_{aoy})=\overline{(e_{aoy},x)}=\overline{a(y,x)}=\overline{a}(x,y)$.

4- For every $u \in aox$ we have $e_{a^{-1}ou} \in 1ox$ and so $(e_{a^{-1}ou},e_{a^{-1}ou}) \leq (x,x)$. This by Part 3 implies that $ \frac{1}{a^2}(u,u) \leq (x,x)$.
\end{proof}
\begin{theorem}
 If $X$ is a hyperinner product space, then $X$ is normal.
\end{theorem}
\begin{proof}
Let $x_1,x_2\in X$ and $a \in F$. By Parts 3 and 5 of Definition 3.3, for every $y\in X$ we have
 \begin{eqnarray*}
(e_{aox_1}+e_{aox_2},y) &=&a(x_1,y)+a(x_2,y)\\
&=& a(x_1+x_2,y)\\
&=& (e_{ao(x_1+x_2)},y)
\end{eqnarray*}
which implies that $(e_{aox_1}+e_{aox_2}-e_{ao(x_1+x_2)} ,y)=0$. The arbitraryness of $y$ yields $e_{aox_1}+e_{aox_2}=e_{ao(x_1+x_2)}$. To prove the correctness of the second condition in Lemma 2.12 is similar. So the proof is complete.
\end{proof}
\begin{lemma}
Let $f:X \rightarrow \mathbb{R}$ be a mapping such that for every $x\in X$, $f(x)=\sqrt{(x,x)}$. Then
the following properties are hold.

1- $  $$f(x)=0 \Leftrightarrow x=0$

2- $  $$\vert(x,y)\vert\leq f(x).f(y)$

3- $  $$f(x+y)\leq f(x)+f(y)$

4- $  $$f(e_{aox})=\vert a \vert f(x)$

5- $  $$\mathrm{sup }f(aox)=\vert a \vert f(x)$.
\end{lemma}
\begin{proof}1- Obvious.

2- For $x=0$, the relation is hold. For $x\neq0$, let
$a=\frac{(y,x)}{f^2(x)}$. Thus we have
\begin{eqnarray*}
x & \leq & (y-e_{aox},y-e_{aox})=(y,y)-(y,e_{aox})-(e_{aox},y)+(e_{aox},e_{aox}) \\
&=& f^2(y)-a(x,y) \\
&=& f^2(y)- \frac{\mid(x,y)\mid^2}{f^2(x)}
\end{eqnarray*}
which implies the assertion.

3- By 2 we have
\begin{eqnarray*}
f^2(x+y) &=&(x+y,x+y)\\
&=& f^2(x)+f^2(y)+(x,y)+(y,x)\\
& \leq & (f(x)+f(y))^2
\end{eqnarray*}
which implies the assertion.

4- $ $$f^2(e_{aox})=(e_{aox},e_{aox})=a\overline{a}(x,x)=\vert a
\vert^2f^2(x)$.

5- By Part 6 of Definition 3.3 we have $f(u) \leq \vert a \vert f(x)$ for every $u \in aox$. This implies that $\mathrm{sup }f(aox) \leq \vert a \vert f(x)=f(e_{aox})$. Since $f(e_{aox}) \leq \mathrm{sup }f(aox)$, we obtain $\mathrm{sup }f(aox)=\vert a \vert f(x)$.
\end{proof}
\begin{definition}
$[14]$ A norm on $X$ is a mapping $\|.\| :X \longrightarrow
\mathbb{R}$ such that for every $a\in F$ and $x,y\in X$ we have

1- $  $$\|x\|=0 \Leftrightarrow x=0$

2- $  $$\|x+y\|\leq \|x\|+\|y\| $

3- $  $$ \mathrm{sup} \|aox\|=\vert a \vert\|x\| $.

 $X$ with a norm is called a normed hypervector space.
\end{definition}
\begin{corollary}
The defined mapping in Lemma 3.6 is a norm on $X$.
\end{corollary}
\begin{proof}
Part 1, 3 and 4 of Lemma 3.6 together with Definition 3.7 follows the assertion.
\end{proof}
\par \vspace{.4cm}{\bf Acknowledgements.} This research is partially
Supported by the Research Center in Algebraic Hyperstructures and
Fuzzy Mathematics, University of Mazandaran, Babolsar, Iran.
\bibliographystyle{amsplain}

\begin{thebibliography}{10}\large


\bibitem{co}  P. Corsini, \textit{Prolegomena of hypergroup theory}, Aviani editore, (1993).
\bibitem{co}  P. Corsini and V. Leoreanu, \textit {Applications of
Hyperstructure theory}, Kluwer Academic Publishers, Advances in
Mathematics (Dordrecht), (2003).

\bibitem{ma}  F. Marty, \textit{Sur nue generalizeation de la
notion de group}, $8^{th}$ congress of the Scandinavic
Mathematics, Stockholm, (1934), 45-49.

\bibitem{ma}  P. Raja, S. M. Vaezpour, \textit{Normed hypervector spaces}, Iranian Journal of Mathematical Sciences and Informatics, Vol. 2, No. 2 (2007), 35-44.

\bibitem{ma}  S. Roy, T.K. Samanta, \textit{Innerproduct hyperspaces}, Accepted in Italian J. of
Pure and Appl. Math.

\bibitem{ta}  A. Taghavi, R. Hosseinzadeh,
\textit{A note on dimension of weak hypervector spaces}, Accepted in Italian J. of
Pure and Appl. Math.

\bibitem{ta}  A. Taghavi, R. Hosseinzadeh,
\textit{Hahn-Banach Theorem for functionals on hypervector
spaces}, The Journal of Mathematics and Computer Science, Vol .2 No.4 (2011) 682-690.

\bibitem{ta}  A. Taghavi, R. Hosseinzadeh,
\textit{Operators on normed hypervector spaces}, Southeast Asian
Bulletin of Mathematics, (2011) 35: 367-372.

\bibitem{ta}  A. Taghavi, R. Hosseinzadeh,
\textit{Operators on weak hypervector spaces}, Ratio Mathematica, 22 (2012) 37-43

\bibitem{ta}  A. Taghavi, R. Hosseinzadeh,
\textit{Uniform Boundedness Principle for operators on
hypervector spaces}, Iranian Journal of Mathematical Sciences and Informatics,
Vol. 7, No. 2 (2012) 9-16

\bibitem{ta}  A. Taghavi, R. Parvinianzadeh,
\textit{Hyperalgebras and Quotient Hyperalgebras},  Italian J. of
Pure and Appl. Math, No. 26 (2009) 17-24

\bibitem{ta}  A. Taghavi, T. Vougiouklis, R. Hosseinzadeh,
\textit{A note on Operators on Normed Finite Dimensional Weak Hypervector
Spaces }, Scientific bulletin, Series A, Vol. 74, Iss. 4 (2012) 103-108.

\bibitem{ta}  M. Scafati-Tallini,
\textit{Characterization of remarkable Hypervector space}, Proc.
$8^{th}$ congress on "Algebraic Hyperstructures and Aplications",
Samotraki, Greece, (2002), Spanidis Press, Xanthi, (2003),
231-237.


\bibitem{ta}  M. Scafati-Tallini, \textit{Weak Hypervector space and
norms in such spaces}, Algebraic Hyperstructures and Applications
Hadronic Press. (1994), 199--206.


\bibitem{vo1}  T. Vougiouklis, \textit{The fundamental relation
in hyperrings. The general hyperfield. Algebraic hyperstructures
and applications (Xanthi, 1990)}, World Sci. Publishing, Teaneck,
NJ, (1991), 203--211.


\bibitem{vo2}  T. Vougiouklis, \textit{ Hyperstructures and their
representations}, Hadronic Press, (1994).

\bibitem{vo2}  M. M. Zahedi, \textit{ A review on hyper k-algebras}, Iranian Journal of Mathematical Sciences and Informatics, Vol. 1, No. 1 (2006), 55-112

\end{thebibliography}

\end{document}